\documentclass[a4paper,leqno,11pt]{amsart}

\usepackage{amsfonts,amssymb,verbatim,amsmath,amsthm,latexsym,textcomp,amscd}
\usepackage{latexsym,amsfonts,amssymb,epsfig,verbatim}
\usepackage{amsmath,amsthm,amssymb,latexsym,graphics,textcomp}
\usepackage{graphicx}
\usepackage{color}
\usepackage{url}
\usepackage{enumerate}
\usepackage[mathscr]{euscript}

\input xy
\xyoption{all}

\usepackage{hyperref}

\theoremstyle{plain}
\newtheorem{theorem}{Theorem}[section]

\newtheorem{lemma}[theorem]{Lemma}
\newtheorem{corollary}[theorem]{Corollary}

\theoremstyle{definition}
\newtheorem{definition}[theorem]{Definition}

 \newtheorem{subsec}[theorem]{}

 2
\newcommand{\ustar}{\hbox{\ $\cup$\hskip -5.5 pt\raise 1 pt\hbox{$*$}}}

\newcommand{\ntos}{{}^{\hbox{$\ben$}}\!S}
\newcommand{\domain}{\hbox{\rm domain}}

\newcommand{\emp}{\emptyset}

\newcommand{\ben}{\mathbb N}
\newcommand{\beq}{\mathbb Q}
\newcommand{\seqx}{\langle x_n\rangle_{n=1}^\infty}
\newcommand{\vvarphi}{\raise 2 pt \hbox{$\varphi$}}

\newcommand{\nhat}[1]{\{1,2,\ldots,#1\}}

\begin{document}

\title{Combinatorially rich sets in partial semigroups}
\author{Arpita Ghosh}
\address{Department of Mathematics,
          University of Haifa,
          3498838 Haifa, Israel.}
\email{arpi.arpi16@gmail.com}
\author{Neil Hindman}
\address{Department of Mathematics,
                 Howard University,
                  Washington, DC 20059, USA}
\email{nhindman@aol.com}


\begin{abstract} There are several notions of size for semigroups
that have natural analogues for partial semigroups. Among these are
{\it thick\/}, {\it syndetic\/}, {\it central\/}, {\it piecewise syndetic\/},
{\it IP\/}, {\it J\/}, and the more recently 
introduced notion of {\it combinatorially rich\/}, abbreviated CR.  
We investigate the notion of CR set for adequate partial semigroups,
its relation to other notions, especially J sets, and some surprising 
differences among them.
\end{abstract}
\maketitle

\section{Notions of size}

If $(S,\cdot,{\mathcal U})$ is a compact Hausdorff right topological semigroup
(meaning that $(S,{\mathcal U})$ is a compact Hausdorff topological space, $(S,\cdot)$ is
a semigoup), and for each $x\in S$, the function $\rho_x:S\to S$ defined by 
$\rho_x(y)=y\cdot x$ is continuous), then $S$ has idempotents and a smallest 
two sided ideal. The smallest ideal is denoted by $K(S)$. If $(S,\cdot)$ is a discrete semigroup and
$(\beta S,{\mathcal U})$ is its Stone-\v Cech compactification, there is a unique extension
of the operation to $\beta S$ so that $(\beta S,\cdot,{\mathcal U})$
is a compact right topological semigroup with $S$ contained in its
topological center (meaning that for each $x\in S$ the function
$\lambda_x:\beta S\to\beta S$ defined by $\lambda_x(y)=x\cdot y$ is continuous).
We take the points of $\beta S$ to be ultrafilters on $S$, identifying 
the principal ultrafilters with the points of $S$. For an elementary
introduction to the algebra of $\beta S$ see \cite[Part I]{HS}.

All of the notions of size mentioned in the abstract except J and CR have simple
algebraic descriptions in semigroups which we take here as the definitions.

\begin{definition}\label{defalg} Let $(S,\cdot)$ be a discrete semigroup and
let $A\subseteq	S$. \begin{itemize}
\item[(1)] The set $A$ is a {\it thick \/} set if and only if there is a left ideal
$L$ of $\beta S$ such that $L\subseteq \overline A$.
\item[(2)] The set $A$ is a {\it syndetic \/} set if and only if for every left ideal
$L$ of $\beta S$, $\overline A\cap L\neq\emp$.
\item[(3)] The set $A$ is a {\it central \/} set if and only if there is an
idempotent\\
 $p\in K(\beta S)\cap\overline A$.
\item[(4)] The set $A$ is a {\it piecewise syndetic \/} (PS) set if and only if 
$\overline A\cap K(\beta S)\neq\emp$.
\item[(5)] The set $A$ is an {\it IP \/} set if and only if there is an
idempotent $p\in\overline A$.
\end{itemize}\end{definition}

All of the notions listed in Definition \ref{defalg} have elementary
characterizations. But the characterization for central is very 
complicated, so we will only refer to it as \cite[Theorem 14.25]{HS}.
We list the others now.

Given a set $X$, we let $\mathcal{P}_f(X)=\{H:H$ is a finite nonempty subset of $X\}.$
Given a sequence $\langle x_n\rangle_{n=1}^\infty$ in $S$, we let
$FP(\langle x_n\rangle_{n=1}^\infty)=\{\prod_{t\in H}x_t:H\in\mathcal{P}_f(\ben)\}$ where
the product $\prod_{t\in H}x_t$ is computed in increasing order of indices. 
Given $A\subseteq S$ and $x\in S$, $x^{-1}A=\{y\in S:x\cdot y\in A\}$.

\begin{lemma}\label{lemcomb}
Let $(S,\cdot)$ be a discrete semigroup and
let $A\subseteq	S$. \begin{itemize}
\item[(1)] The set $A$ is a thick set if and only if for each
$F\in\mathcal{P}_f(S)$ there exists $x\in S$ such that $F\cdot x\subseteq A$.
\item[(2)] The set $A$ is a syndetic set if and only if there exists
some $G\in\mathcal{P}_f(S)$ such that $S=\bigcup_{t\in G}t^{-1}A$.
\item[(3)] The set $A$ is a piecewise syndetic set if and only if 
there exists some $G\in\mathcal{P}_f(S)$ such that for every $F\in\mathcal{P}_f(S)$
there exists $x\in S$ such that $F\cdot x\subseteq \bigcup_{t\in G}t^{-1} A$.
\item[(4)] The set $A$ is an IP set if and only if there exists a sequence
$\langle x_n\rangle_{n=1}^\infty$ in $S$ such that
$FP(\langle x_n\rangle_{n=1}^\infty)\subseteq A$.
\end{itemize}
\end{lemma}

\begin{proof} Proofs of the equivalence of the characterzations of
Definition \ref{defalg} and those given here for thick, syndetic,
piecewise syndetic, and IP are in \cite[Theorem 4.48(a)]{HS},
\cite[Theorem 4.48(b)]{HS}, \cite[Theorem 4.40]{HS}, and \cite[Theorem 5.12]{HS}
respectively.
 \end{proof}

The notion of a {\it combinatorially rich\/}  set was
introduced for commutative semigroups (using a definition
equivalent to Lemma \ref{lemJCRcom}(2) below) by Bergelson and
Glasscock in \cite{BG}. We write $\ntos$ for the set of infinite sequences in $S$,
i.e., the set of functions from $\ben$ to $S$. For $k\in\ben$
we write $\mathcal{P}_f(\ntos)_{\leq k}=\{F\in \mathcal{P}_f(\ntos):|F|\leq k\}$.

\begin{definition}\label{defJCR} Let $(S,\cdot)$ be a discrete semigroup and
let $A\subseteq	S$.
\begin{itemize} 
\item[(1)] The set $A$ is a {\it J\/} set if and only if for each $F\in\mathcal{P}_f(\ntos)$
there exist $m\in \ben$, $a\in S^{m+1}$, and $t(1)<t(2)<\ldots<t(m)$ in $\ben$
such that for each $f\in F$, $a(1)\cdot f(t(1))\cdot a(2)\cdots
a(m)\cdot f(t(m))\cdot a(m+1)\in A$.
\item[(2)] The set $A$ is a {\it combinatorially rich\/} set (CR set) if and only if for each $k\in\ben$ 
there exists $r\in\ben$ such that for each $F\in\mathcal{P}_f(\ntos)_{\leq k}$
there exist $m\in \ben$, $a\in S^{m+1}$, and $t(1)<t(2)<\ldots<t(m)\leq r$ in $\ben$
such that for each $f\in F$, $a(1)\cdot f(t(1))\cdot a(2)\cdots
a(m)\cdot f(t(m))\cdot a(m+1)\in A$.
\end{itemize}
\end{definition}

In the event $S$ is commutative, the notions J and CR are simpler.

\begin{lemma}\label{lemJCRcom}Let $(S,\cdot)$ be a discrete commutative semigroup and
let $A\subseteq	S$.
\begin{itemize}
\item[(1)] The set $A$ is a J set if and only if for each $F\in\mathcal{P}_f(\ntos)$
there exist $a\in S$ and $H \in \mathcal{P}_f(\ben)$ such that for each $f\in F$, $a\cdot \prod_{t\in H} f(t)\in A$.
\item[(2)] The set $A$ is a CR set if and only for each $k\in  \ben$
there exists $r\in\ben$ such that for each $F\in\mathcal{P}_f(\ntos)_{\leq k}$
there exist $a\in S$ and $H\in\mathcal{P}_f(\nhat{r})$ such that for all
$f\in F$, $a\cdot\prod_{t\in H}f(t)\in A$.
\end{itemize}
\end{lemma}

\begin{proof} (1) \cite[Lemma 14.14.2]{HS}.

(2) \cite[Lemmas 1.3 and 2.3]{HHST}.\end{proof}

The implications in Figure 1 were established in
\cite[Section 4]{H} except for the ones involving CR.
The fact that CR implies J is immediate from the definitions;
the fact that PS implies CR is \cite[Theorem 3.3]{HHST}.

\setlength{\unitlength}{20 pt}

\begin{picture}(16,11)(1,0)
\put(9.3,10.1){Thick}
\put(9.3,8.1){Central}
\put(9.8,6.1){IP}
\put(9.8,8){\vector(-1,-1){1.25}}
\put(7.1,8.1){Syndetic}
\put(7.8,6.1){PS}
\put(7.8,4.1){CR}
\put(7.9,2.1){J}
\put(10,10){\vector(0,-1){1.25}}
\put(10,8){\vector(0,-1){1.25}}
\put(8,8){\vector(0,-1){1.25}}
\put(8,6){\vector(0,-1){1.25}}
\put(8,4){\vector(0,-1){1.25}}
\put(4,1){Figure 1: Implications for semigroups.}
\end{picture}

\section{Partial semigroups}

A {\it partial semigroup\/} is a pair $(S,*)$ where $S$ is a nonempty set and
$*$ is an operation defined on a nonempty subset of $S\times S$
such that for all $x,y,z\in S$, $(x*y)*z=x*(y* z)$ in the sense that 
if either side is defined, then so is the other and they are equal.

If $(S,*)$ is a discrete partial semigroup, it turns out that there 
is a subset $\delta S$ of $\beta S$ and there is an extension of the operation
to $\delta S$ making $(\delta S,*)$ a compact right topological 
semigroup with all of the relevant algebraic structure including
a smallest ideal $K(\delta S)$. The only requirement is that the partial
semigroup be {\it adequate\/}.

\begin{definition}\label{defadeq}Let $(S,*)$ be a partial semigroup.
\begin{itemize}
\item[(1)] For $a\in S$, $\vvarphi(a)=\{b\in S:a*b$ is defined$\}$.
\item[(2)] For $F\in\mathcal{P}_f(S)$, $\sigma(F)=\bigcap_{a\in F}\vvarphi(a)$.
\item[(3)] The partial semigroup $(S,*)$ is {\it adequate\/} if and 
only if for every $F\in\mathcal{P}_f(S)$, $\sigma(F)\neq\emp$.
\item[(4)] If $(S,*)$ is adequate, then $\delta S=\bigcap_{F\in\mathcal{P}_f(S)}\overline{\sigma(F)}$.
\end{itemize}
\end{definition}

\begin{theorem}\label{thmdelta} Let $(S,*)$ be an adequate partial semigroup.
Then $(\delta S,*)$ is a compact Hausdorff right topological semigroup.
\end{theorem}

\begin{proof} \cite[Theorem 2.10]{HM} \end{proof}

Experience has shown that most naturally arising partial semigroups
are adequate. An obvious counterexample is the set of finite
matrices over a ring under multiplication.

We shall want to use the following examples.

\begin{lemma} \label{defparexamples}\begin{itemize}
\item[(1)] Let $X$ be an infinite set, let $S=\mathcal{P}_f(X)$, and for 
$F,G\in S$, let $F\ustar G=F\cup G$ defined if and only if $F\cap G=\emp$.
Then $(S,\ustar)$ is an adequate partial semigroup.
\item[(2)] Let $(X,\leq)$ be a linearly ordered infinite set without a largest element,
let $S=\mathcal{P}_f(X)$, and for $F,G\in S$, let $F\uplus G=F\cup G$ defined if and only if $\max F<\min G$.
Then $(S,\uplus)$ is an adequate partial semigroup.
\item[(3)] Let $\Sigma$ be a nonempty set, let $L(\Sigma)=\{f: F\in\mathcal{P}_f(\ben)$ and $f:F\to\Sigma\}$, 
and let $S=L(\Sigma)$. For $f,g\in S$, define $f*g=f\cup g$ defined if and only if
$\max\domain(f)<\min\domain(g)$. Then $(S,*)$ is an adequate partial semigroup.
\end{itemize}\end{lemma}

\begin{proof} These are all easy exercises. \end{proof}

The top five notions from Figure 1 are all defined in a nearly
identical fashion for adequate partial semigroups.

\begin{definition}\label{defparalg} Let $(S,*)$ be a discrete adequate partial semigroup and
let $A\subseteq	S$. \begin{itemize}
\item[(1)] The set $A$ is a {\it thick \/} set if and only if there is a left ideal
$L$ of $\delta S$ such that $L\subseteq \overline A$.
\item[(2)] The set $A$ is a {\it syndetic \/} set if and only if for every left ideal
$L$ of $\delta S$, $\overline A\cap L\neq\emp$.
\item[(3)] The set $A$ is a {\it central \/} set if and only if there is an
idempotent\\
 $p\in K(\delta S)\cap\overline A$.
\item[(4)] The set $A$ is a {\it piecewise syndetic \/} set if and only if 
$\overline A\cap K(\delta  S)\neq\emp$.
\item[(5)] The set $A$ is an {\it IP \/} set if and only if there is an
idempotent $p\in\overline A\cap \delta S$.
\end{itemize}\end{definition}

\begin{picture}(16,7)(1,0)
\put(9.3,6.1){Thick}
\put(9.3,4.1){Central}
\put(9.8,2.1){IP}
\put(9.8,4){\vector(-1,-1){1.25}}
\put(7.1,4.1){Syndetic}
\put(7.8,2.1){PS}
\put(10,6){\vector(0,-1){1.25}}
\put(10,4){\vector(0,-1){1.25}}
\put(8,4){\vector(0,-1){1.25}}
\put(3,1){Figure 2: Implications for adequate partial semigroups.}
\end{picture}

\begin{theorem}\label{thmFigure2} Let $(S,*)$ be an adequate partial 
semigroup. All of the implications in Figure {\rm 2} are valid for subsets
of $S$. \end{theorem}

\begin{proof} That central implies both IP and PS is immediate from
the definitions. That syndetic implies PS follows from the fact 
that $K(\delta S)$ is a left ideal of $\delta S$. To see that thick
implies central, let $A$ be a thick subset of $S$ and pick a left ideal
$L$ of $\delta S$ such that $L\subseteq \overline A$.
By \cite[Corollary 2.6]{HS}, $L$ contains a minimal left ideal of $\delta S$ which
has an idempotent and by \cite[Theorem 2.8]{HS} this minimal
left ideal is contained in $K(\delta S)$.
\end{proof}

When we direct our attention to the combinatorial characterizations of\break
 Lemma \ref{lemcomb}, the situation changes dramatically. In her
dissertation \cite{M}, Jillian McLeod considered the analogous combinatorial
definitions of thick, syndetic, piecewise syndetic, and IP. (For
thick and piecewise syndetic one requires that $x\in\sigma(F)$, for 
syndetic and piecewise syndetic one defines $t^{-1}A=\{s\in\vvarphi(t):t*s\in A\}$,
and for $FP(\langle x_n\rangle_{n=1}^\infty)$ one requires that the products are 
defined. She showed that in none of these cases was the combinatorial
version equivalent to the algebraic definition!

We have another shock in store, but first we need to define J sets and 
CR sets in adequate partial semigroups, beginning with
J sets where we follow what was done in \cite{HP}. Since the definition
depends on sequences in $S$, it is clear that we want sequences whose
products are defined. 

\begin{definition}\label{defadseq} Let $(S,*)$ be an adequate partial 
semigroup.
\begin{itemize}
\item[(1)] A sequence $\langle f(t)\rangle_{t=1}^\infty$ in $S$
is {\it adequate\/} if and only if 
\begin{itemize}
\item[(i)] for each $H\in\mathcal{P}_f(\ben)$, $\prod_{t\in H}f(t)$ is defined and
\item[(ii)] for each $F\in\mathcal{P}_f(S)$, there exists $m\in\ben$ such that\\
$FP(\langle f(t)\rangle_{t=m}^\infty)\subseteq \sigma(F)$.
\end{itemize}
\item[(2)] The set of all adequate sequences in $S$ is denoted by
${\mathcal T}_S$ or just ${\mathcal T}$.\end{itemize}\end{definition}.
  
The justification given in \cite{HP} for item (ii) in the definition of
adequate sequences was that the proofs demanded it.  In defense of the 
choice, the central sets theorem for adequate partial semigroups,
\cite[Theorem 3.6]{HP}, is verbatim the same as the central sets theorem
for semigroups except that ``adequate sequences'' replaced ``sequences''.
Further, letting 
$$J(S)=\{p\in\delta S:(\forall A\in p)(A\hbox{ \rm  is a J set})\}\,,$$
one obtains in \cite[Corollary 3.4]{HP} that $J(S)$ is a two sided ideal of 
$\delta S$ in direct analogy with the situation for semigroups in $\beta S$.

It was shown in \cite[Theorem 4.1]{HP} that if $S$ is countable, then
adequate sequences are plentiful. We shall have more to say about this
subject in Section \ref{secsimpler}.

We are now in a position to define J sets for adequate partial 
semigroups. 

\begin{definition}\label{defJpartial} Let $(S,*)$ be an adequate partial 
semigroup and let $A\subseteq S$. The set $A$ is a J set if and only if\\
$(\forall L\in\mathcal{P}_f(S))(\forall F\in\mathcal{P}_f(\mathcal{T}))(\exists m\in\ben)
(\exists a\in S^{m+1})(\exists t(1)<\ldots <t(m)\hbox{ in }\ben)\\
(\forall f\in F)
(a(1)*f(t(1))*a(2)*\ldots *a(m)*f(t(m))*a(m+1)\in A\cap\sigma(L))$.
\end{definition}

Based on the definition of J set, we define a CR set. For $k\in\ben$, we write
$\mathcal{P}_f(\mathcal{T})_{\leq k}=\{F\in\mathcal{P}_f(\mathcal{T}):0<|F|\leq k\}$.

\begin{definition}\label{defCRpartial} Let $(S,*)$ be an adequate partial 
semigroup and let $A\subseteq S$. The set $A$ is a CR set if and only if\\
$(\forall k\in\ben)(\forall L\in\mathcal{P}_f(S))
(\exists r\in\ben)\\
(\forall F\in\mathcal{P}_f(\mathcal{T})_{\leq k})(\exists m\in\ben)
(\exists a\in S^{m+1})(\exists t(1)<\ldots <t(m)\leq r\hbox{ in }\ben)\\
(\forall f\in F)
(a(1)*f(t(1))*a(2)*\ldots *a(m)*f(t(m))*a(m+1)\in A\cap\sigma(L))$.
\end{definition}

One easily sees that CR sets in adequate partial semigroups are J sets, and
one would like to show that piecewise syndetic sets must be CR sets.  
But it is not true! In fact, alone among the notions of largeness for 
semigroups (that we know of) $S$ need not even be large.

Our example does not even satisfy the weakest version of CR which we introduce now

\begin{definition}\label{defkcr} Let $(S,*)$ be an adequate partial semigroup and let $A\subseteq S$.
For $k\in\ben$, $A$ is a $k$-CR set if and only if 
$(\forall L\in\mathcal{P}_f(S))(\exists r\in\ben)\\
(\forall F\in\mathcal{P}_f(\mathcal{T})_{\leq k})(\exists m\in\ben)
(\exists a\in S^{m+1})(\exists t(1)<\ldots <t(m)\leq r\hbox{ in }\ben)\\
(\forall f\in F)
(a(1)*f(t(1))*a(2)*\ldots* a(m)*f(t(m))*a(m+1)\in A\cap\sigma(L))$.
\end{definition}

Thus $A$ is a CR set if and only if $A$ is a $k$-CR set for every $k\in\ben$.

\begin{theorem}\label{Snotcr} There is a countable adequate partial semigroup
$(S,*)$ such that $S$ is not a $1$-CR set. In particular piecewise syndetic
does not imply CR, not even $1$-CR. \end{theorem}

\begin{proof} Let $\beq^+=\{x\in\beq:x>0\}$ and
let $S=\mathcal{P}_f(\beq^+)$ with the operation $\uplus$ defined in Definition
\ref{defparexamples}.

 For each $n\in\ben$, define $f_n:\ben\to S$ by 
$f_n(t)=\left\{\begin{array}{cl}
\{1-\frac{1}{t}\}&\hbox{if }t\leq n\\ 
\{t\}&\hbox{if }t>n\,.\end{array}\right.$\\
  Then each $f_n$ is an adequate sequence in $S$.

Suppose that $S$ is a 1-CR set in $S$. Let $L=\big\{\{1\}\big\}$ and note that\\
$\sigma(L)=\vvarphi(\{1\})=\{H\in S:\min H>1\}$.  
Pick $r\in\ben$ such that\\
 $(\forall f\in {\mathcal T})(\exists m\in\ben)
(\exists a\in S^{m+1})(\exists t(1)<\ldots <t(m)\leq r\hbox{ in }\ben)\\
(a(1)\uplus f(t(1))\uplus a(2)\uplus\ldots\uplus a(m)\uplus 
f(t(m))\uplus a(m+1)\in \sigma(L))$.

Let $f=f_r$ and pick $m\in\ben$, $a\in S^{m+1}$, and $t(1)<\ldots <t(m)\leq r\hbox{ in }\ben$
such that
$a(1)\uplus f_r(t(1))\uplus a(2)\uplus\ldots\uplus a(m)\uplus 
f_r(t(m))\uplus a(m+1)\in \sigma(L)$.\\
Then $a(1)\in\sigma(L)$ so $\min a(1)>1$  and this implies $f_r(t(1))\in\sigma(L\uplus a(1))$
so $\min f_r(t(1))>\max a(1)> 1$. But $t(1)\leq r$ so $\min f_r(t(1))=1-\frac{1}{t(1)}<1$,
a contradiction.

It is trrivial that any adequate partial semigroup is piecewise syndetic in itself.
\end{proof}

The fact that the above example is countable is important because there 
are many things that we only know hold in countable partial semigroups.

\section{When  $S$ is a CR set}

By virtue of Theorem \ref{Snotcr} we are interested 
in finding out when $S$ is guaranteed to be a CR set, and
what other structure is then guaranteed.

\begin{definition} Let $(S,*)$ be an adequate partial semigroup. Then $(S,*)$ has property
(\dag) if and only if $(\forall k\in\ben)(\forall L\in\mathcal{P}_f(S))(\exists r\in\ben)
(\forall F\in\mathcal{P}_f(\mathcal{T})_{\leq k})\\
(\exists t\in\nhat{r})(\forall f\in F)(f(t)\in\sigma(L))$.
\end{definition}

\begin{theorem}\label{thmdag} Let $(S,*)$ be an adequate partial semigroup which has property (\dag).
Then $S$ is a CR set in $S$.
\end{theorem}

\begin{proof} Let $k\in\ben$ and $L\in\mathcal{P}_f(S)$ be given. Pick $a\in\sigma(L)$ and 
let $M=L*a$. Pick $r$ as guaranteed for $k$ and $M$ by (\dag). Let $F\in\mathcal{P}_f({\mathcal T})_{\leq k}$.
Pick $t\in\nhat{r}$ such that for all $f\in F$, $f(t)\in\sigma(M)$.
Let $P=\{x*a*f(t):x\in L$ and $f\in F\}$. Let $m=1$, $a(1)=a$, and pick $a(2)\in \sigma(P)$.
Let $t(1)=t$.
Then for $f\in F$, $a(1)*f(t(1))*a(2)\in \sigma(L)$.
\end{proof}

\begin{lemma} Let $X$ be an infinite set. The partial semigroup $(\mathcal{P}_f(X),\ustar)$ 
has property (\dag).
\end{lemma}

\begin{proof} Let $S=\mathcal{P}_f(X)$. Let $k\in\ben$ and $L\in\mathcal{P}_f(S)$ be given. 
Let $d=|\bigcup L|$ and let $r=kd+1$. Let $F\in\mathcal{P}_f(\mathcal{T})_{\leq k}$ be given
and enumerate $F$ as $\{f_1,f_2,\ldots,f_k\}$, with repetition if necessary.
For $j\in\nhat{k}$, let $B_j=\{t\in\nhat{r}:f_j(t)\cap(\bigcup L)\neq\emp\}$.
Since $f_j(1),f_j(2),\ldots,f_j(r)$ are pairwise disjoint, we get 
$|B_j|\leq d$. Consequently $|\bigcup_{j=1}^kB_j|\leq kd$ so pick $t\in\nhat{r}\setminus
\bigcup_{j=1}^kB_j$. Then for all $j\in\nhat{k}$, $f_j(t)\cap(\bigcup L)=\emp$
so $f_j(t)\in\sigma(L)$.
\end{proof}

Note in particular that $(\mathcal{P}_f(\beq^+),\ustar)$ has property (\dag), so is a CR set
while we saw in Theorem \ref{Snotcr} that $(\mathcal{P}_f(\beq^+),\uplus)$ is not a 1-CR set.

Following is a more simply stated sufficient condition.

\begin{definition} Let $(S,*)$ be an adequate partial semigroup. Then $(S,*)$ has property
(\ddag) if and only if $(\forall L\in\mathcal{P}_f(S))(\exists r\in\ben)
(\forall f\in{\mathcal T})(f(r)\in\sigma(L))$.\end{definition}

Since (\ddag) implies (\dag), it is also a sufficient condition for $S$ to be
a CR set. 
Note that both $(\mathcal{P}_f(\ben),\uplus)$ and $(L(\Sigma),*)$ have property (\ddag).

\begin{definition}  Let $(S,*)$ be an adequate partial semigroup.
A subset $A$ of $S$ is {\it {\v c}-piecewise syndetic\/} in $S$ if and only if
$(\exists H\in\mathcal{P}_f(S))(\forall T\in\mathcal{P}_f(S))\\
(\exists x\in\sigma(T))
((T\cap\sigma(H))*x\subseteq\bigcup_{s\in H}s^{-1}A)$.
\end{definition}

We have seen that piecewise syndetic need not imply CR.

\begin{theorem}\label{psimplcr} Let $(S,*)$ be an adequate partial semigroup
which is a $1$-CR set in $S$ and let $A$ be a piecewise syndetic subset of
$S$. Then A is a $1$-CR set in $S$.\end{theorem}

\begin{proof} To see that $A$ is a 1-CR set in $S$, let $L\in\mathcal{P}_f(S)$. We need to show that
$(\exists r\in\ben)(\forall f\in{\mathcal T})(\exists m\in\ben)
(\exists a\in S^{m+1})(\exists t(1)<\ldots <t(m)\leq r\hbox{ in }\ben)\\
(a(1)*f(t(1))*a(2)*\ldots * a(m)*f(t(m))*a(m+1)\in A\cap\sigma(L))$.

We claim that $A\cap\sigma(L)$ is piecewise syndetic in $S$. To see this, 
pick $p\in \overline A\cap K(\delta S)$. Since $\delta S\subseteq \overline{\sigma(L)}$
we have that $p\in \overline{A\cap\sigma(L)}\cap K(\delta S)$.

By \cite[Theorem 3.10]{M}, $A\cap\sigma(L)$ is {\v c}-piecewise syndetic so pick
$H\in\mathcal{P}_f(S)$ such that $(\forall T\in\mathcal{P}_f(S))(\exists x\in\sigma(T))
((T\cap\sigma(H))*x\subseteq\bigcup_{s\in H}s^{-1}(A\cap\sigma(L)))$.

Since $S$ is 1-CR in $S$ and $H\in\mathcal{P}_f(S)$, pick $r\in\ben$ such that\\
$(\forall f\in{\mathcal T})(\exists m\in\ben)
(\exists a\in S^{m+1})(\exists t(1)<\ldots <t(m)\leq r\hbox{ in }\ben)\\
(a(1)*f(t(1))*a(2)*\ldots* a(m)*f(t(m))*a(m+1)\in \sigma(H))$.

We claim that $r$ is as required to show that $A$ is a 1-CR set. So let 
$f\in{\mathcal T}$ be given. Pick $m\in\ben$, $a\in S^{m+1}$, and
$t(1)<\ldots<t(m)\leq r$ in $\ben$ such that 
$a(1)*f(t(1))*a(2)*\ldots* a(m)*f(t(m))*a(m+1)\in \sigma(H))$.

Let $T=\{a(1)*f(t(1))*a(2)*\ldots * a(m)*f(t(m))*a(m+1)\}$.
Then $T\subseteq\sigma(H)$ so pick $x\in\sigma(T)$ such that $T*x\subseteq
\bigcup_{s\in H}s^{-1}(A\cap \sigma(L))$. Pick $s\in H$ such that
$s*T*x\subseteq A\cap\sigma(L)$. That is\\ 
$s*a(1)*f(t(1))*a(2)*\ldots * a(m)*f(t(m))*a(m+1)*x\in A\cap\sigma(L)$.
Define $b\in S^{m+1}$ by $b(1)=s*a(1)$, $b(m+1)=a(m+1)*x$ and for $1<j\leq m$, if any, 
$b(j)=a(j)$. Then
$b(1)*f(t(1))*b(2)*\ldots * b(m)*f(t(m))*b(m+1)\in A\cap \sigma(L)$.
\end{proof}

Note that the above proof cannot be simply adapted to prove that if
$S$ is a CR set in $S$ then every piecewise syndetic set in $S$ is
a 2-CR set in $S$ since if $F=\{f,g\}\subseteq {\mathcal T}$, then one
may not be able to pick the same $s\in H$ for both $f$ and $g$.

\begin{definition} Let $(S,*)$ be an adequate partial semigroup.
\begin{itemize}
\item[(1)] $CR(S)=\{p\in\delta S:(\forall A\in p)(A$ is a CR set$)\}$.
\item[(2)] For $k\in\ben$, $k$-$CR(S)=\{p\in\delta S:(\forall A\in p)(A$ is a $k$-CR set$)\}$.
\end{itemize}
\end{definition}

\begin{lemma}\label{crbeta} Let $(S,*)$ be an adequate partial semigroup.
Then for each $k\in\ben$, $k$-$CR(S)=\{p\in\beta S:(\forall A\in p)(A$ is a $k$-CR set$)\}$
and $CR(S)=\{p\in\beta S:(\forall A\in p)(A$ is a CR set$)\}$.
\end{lemma}

\begin{proof} Let $k\in\ben$ and let $p\in \beta S$ such that for all  $A\in p$, $A$ is a $k$-CR set. To see that $p\in\delta S$, let
$L\in\mathcal{P}_f(S)$. Then for each $A\in p$, $A\cap\sigma(L)\neq\emp$ so $\sigma(L)\in p$.
\end{proof}

\begin{theorem}\label{CRideal} Let $(S,*)$ be an adequate partial semigroup. For each
$k\in\ben$, if $k$-$CR(S)\neq\emp$, then $k$-$CR(S)$ is a compact two 
sided ideal of $\delta S$. Consequently, if $CR(S)\neq\emp$, then
$CR(S)$ is a compact two sided ideal of $\delta S$.
\end{theorem}

\begin{proof} Let $k\in\ben$, assume that $k$-$CR(S)\neq\emp$, let 
$p\in k$-$CR(S)$, and let $q\in\delta S$. We will show that $q*p\in k$-$CR(S)$
and $p*q\in k$-$CR(S)$.

Let $A\in p*q$. We need to show that $A$ is a $k$-CR set. So let 
$L\in\mathcal{P}_f(S)$. Let $B=\{x\in S:x^{-1}A\in q\}$. Then $B\in p$ so
pick $r\in\ben$ such that\\
$(\forall F\in\mathcal{P}_f(\mathcal{T})_{\leq k})(\exists m\in\ben)
(\exists a\in S^{m+1})(\exists t(1)<\ldots <t(m)\leq r\hbox{ in }\ben)
(\forall f\in F)\\
(a(1)*f(t(1))*a(2)*\ldots* a(m)*f(t(m))*a(m+1)\in B\cap\sigma(L))$.

To see that $r$ is as required for $A$, let $F\in\mathcal{P}_f(\mathcal{T})_{\leq k}$.
Pick $m\in\ben$, $a\in S^{m+1}$, and $t(1)<\ldots<t(m)\leq r$ in $\ben$ such that
$(\forall f\in F)\\
(a(1)*f(t(1))*a(2)*\ldots* a(m)*f(t(m))*a(m+1)\in B\cap\sigma(L))$.

Let $T=\{a(1)*f(t(1))*a(2)*\ldots* a(m)*f(t(m))*a(m+1):f\in F\}$.
Then $T\subseteq B\cap\sigma(L)$. So $\bigcap_{x\in T}x^{-1}A\in q$. Pick
$y\in \bigcap_{x\in T}x^{-1}A$. Then $y\in\sigma(T)$ and for each 
$f\in F$, $a(1)*f(t(1))*a(2)*\ldots* a(m)*f(t(m))*a(m+1)*y\in A$.
Since $T\subseteq \sigma(L)$, for each $f\in F$, 
$a(1)*f(t(1))*a(2)*\ldots* a(m)*f(t(m))*a(m+1)*y\in \sigma(L)$.
Defiine $b\in S^{m+1}$ by $b(j)=a(j)$ if $j\leq m$ and $b(m+1)=a(m+1)*y$.
Then for each $f\in F$, 
$b(1)*f(t(1))*b(2)*\ldots* b(m)*f(t(m))*b(m+1)\in A\cap\sigma(L)$,
as required.

To see that $q*p\in k$-$CR(S)$, let $A\in q*p$. We need to show that $A$ is a $k$-CR set. So let 
$L\in\mathcal{P}_f(S)$. Let $B=\{x\in S:x^{-1}A\in p\}$. Then $B\in q$ and $\sigma(L)\in q$
so pick $x\in B\cap \sigma(L)$. Then $x^{-1}A\in p$ and $L*x\in \mathcal{P}_f(S)$ so pick $r\in\ben$ such that 
$(\forall F\in\mathcal{P}_f(\mathcal{T})_{\leq k})(\exists m\in\ben)
(\exists a\in S^{m+1})(\exists t(1)<\ldots <t(m)\leq r\hbox{ in }\ben)\\
(\forall f\in F)
(a(1)*f(t(1))*a(2)*\ldots* a(m)*f(t(m))*a(m+1)\in x^{-1} A)$.

To see that $r$ is as required for $A$, let $F\in\mathcal{P}_f(\mathcal{T})_{\leq k}$.
Pick $m\in\ben$, $a\in S^{m+1}$, and $t(1)<\ldots<t(m)\leq r$ in $\ben$ such that
$(\forall f\in F)\\
(a(1)*f(t(1))*a(2)*\ldots* a(m)*f(t(m))*a(m+1)\in x^{-1}A\cap \sigma(L*x))$.
Then for each $f\in F$, $a(1)*f(t(1))*a(2)*\ldots* a(m)*f(t(m))*a(m+1)\in \sigma(L*x)$
so $x*a(1)*f(t(1))*a(2)*\ldots* a(m)*f(t(m))*a(m+1)\in A\cap \sigma(L)$.
Define $b\in S^{m+1}$ by $b(1)=x*a(1)$ and for $j>1$, $b(j)=a(j)$. As before, this completes the proof.
\end{proof} 

\begin{corollary}\label{1crideal} Let $(S,\ast)$ be an adequate partial semigroup which
is a $1$-CR set in $S$. Then $1$-$CR(S)$ is an ideal of $\delta S$.
\end{corollary}

\begin{proof} Theorems \ref{psimplcr} and \ref{CRideal}. \end{proof}

\begin{theorem}\label{CRequiv} Let $(S,*)$ be an adequate partial semigroup. Statements
(1), (2), and (3) are equivalent and are implied by statement (4).
\begin{itemize}
\item[(1)] $CR(S)\neq\emp$.
\item[(2)] Every piecewise syndetic subset of $S$ is a CR set.
\item[(3)] For every $k\in\ben$, every piecewise syndetic subset of $S$ is a 
$k$-CR set.
\item[(4)] $S$ is a CR set in $S$ and whenever
$A_1$ and $A_2$ are subsets of $S$ and $A_1\cup A_2$ is a CR set, then either
$A_1$ or $A_2$ is a CR set.
\end{itemize}
\end{theorem}

\begin{proof} To see that (1) implies (2), assume that $CR(S)\neq\emp$. Then
by Theorem \ref{CRideal}, $CR(S)$ is an ideal of $\delta S$, so
$K(\delta S)\subseteq CR(S)$.

That (2) implies (3) is trivial. To see that (3) implies (1), assume that (3) holds.
We claim that $K(\delta S)\subseteq CR(S)$. Let $p\in K(\delta S)$ and let $A\in p$.
Then for each $k\in\ben$, $A$ is a $k$-CR set so $A$ is a CR set.

To see that (4) implies (1), assume that (4) holds. Let ${\mathcal R}=\{A\subseteq S:
A$ is a CR set in $S\}$. Then ${\mathcal R}$ is partition regular and closed under passage
to supersets so by \cite[Theorem 3.11]{HS} pick $p\in\beta S$ such that
$p\subseteq {\mathcal R}$. Then by Lemma \ref{crbeta}, $p\in\delta S$.
\end{proof}

We know from Theorem \ref{Snotcr} that there exists an adequate partial semigroup $S$ for which $S$ is not a
CR set in which case $CR(S)=\emp$. And we have nontrivial sufficient conditions for $S$ to be a CR set in $S$.

We also know that there exist adequate partial semigroups for which every 
piecewise syndetic set is a CR set and for which the notion of CR set is 
partition regular, namely semigroups.  Unfortunately, we don't have any such examples 
that are not in fact semigroups.

\section{Cartesian products}

There is an obvious meaning of the Cartesian product of 
two partial semigroups.

\begin{definition} Let $(S,*)$ and $(T,*)$ be partial semigroups.
Then $(S\times T,*)$ is the Cartesian product of $S$ and $T$ with the
operation $(a,b)*(c,d)=(a*c,b*d)$ defined if and only if $a*c$ is defined
in $S$ and $b*d$ is defined in $T$.
\end{definition}

\begin{lemma} \label{adeqpartial} Let $(S,*)$ and $(T,*)$ be partial semigroups.
\begin{itemize} 
\item[(1)] $S\times T$ is adequate if and only if $S$ and $T$ are adequate.
\item[(2)] If $f$ is an adequate sequence in $S\times T$, then
$\pi_1\circ f$ is an adequate sequence in $S$ and
$\pi_2\circ f$ is an adequate sequence in $T$.
\item[(3)] If $f$ is an adequate sequence in $S$, 
$g$ is an adequate sequence in $T$, and $h:\ben\to S\times T$ is 
defined by $h(n)=(f(n),g(n))$, then $h$ is an adequate sequence in
$S\times T$.
\end{itemize}
\end{lemma}

\begin{proof} (1) For the sufficiency,  let $F\in \mathcal{P}_f(S\times T)$.
Then\\
 $\emp\neq (\bigcap_{a\in\pi_1[F]}\vvarphi(a))\times(
\bigcap_{b\in\pi_2[F]}\vvarphi(b))\subseteq \bigcap_{(a,b)\in F}\vvarphi((a,b))=\sigma(F)$.

For the necessity, let $F\in\mathcal{P}_f(S)$. Pick $a\in T$.
Then $F\times\{a\}\in\mathcal{P}_f(S\times T)$. If $(x,y)\in\sigma(F\times\{a\})$, then $x\in\sigma(F)$.

Routine calculations establish (2) and (3). \end{proof}

The following proof is very similar to the proof of
\cite[Theorem 4.1]{HHST}.

\begin{theorem}\label{AtimesA} Let $(S,*)$ be an adequate partial semigroup,
let $k\in\ben$, and let $A$ be a $2k$-CR set in $S$. Then $A\times A$
is a $k$-CR set in $S\times S$. In particular, if $A$ is a CR set in $S$,
then $A\times A$ is a CR set in $S\times S$.
\end{theorem}

\begin{proof} Let $L\in\mathcal{P}_f(S\times S)$. Let $M=\pi_1[L]\cup\pi_2[L]$. Pick $r\in\ben$ such
that $(\forall F\in\mathcal{P}_f(\mathcal{T}_S)_{\leq 2k})(\exists m\in\ben)
(\exists a\in S^{m+1})(\exists t(1)<\ldots <t(m)\leq r\hbox{ in }\ben)\\
(\forall f\in F)
(a(1)*f(t(1))*a(2)*\ldots* a(m)*f(t(m))*a(m+1)\in A\cap\sigma(M))$.

Let $F\in\mathcal{P}_f(\mathcal{T}_{S\times S})_{\leq k}$. Let $G=\{\pi_1\circ f:f\in F\}\cup\{\pi_2\circ f:f\in F\}$.
Then $G\in \mathcal{P}_f(\mathcal{T}_S)_{\leq 2k}$.  Pick
$m\in\ben$, $a\in S^{m+1}$, and $t(1)<\ldots <t(m)\leq r$ in $\ben$ such that
for every $h\in G$,\\
 $a(1)*h(t(1))*a(2)*\ldots* a(m)*h(t(m))*a(m+1)\in A\cap\sigma(M)$.

For $i\in\nhat{m+1}$ let $b(i)=(a(i),a(i))$. Then given $f\in F$,\\
$b(1)*f(t(1))*b(2)*\ldots* b(m)*f(t(m))*b(m+1)\in (A\times A)\cap\sigma(L)$.
\end{proof}

We next set out to show that \cite[Proposition 2.3]{BR} holds in the 
partial semigroup $(\mathcal{P}_f(\ben),\uplus)$. It is interesting that even though
it seems not to have been explicitly noted before, $(\mathcal{P}_f(\ben),\uplus)$
is the natural setting for this result rather than the semigroup $(\mathcal{P}_f(\ben),\cup)$.
The reason is, if one is defining an $IP_r$ set in an arbitrary semigroup $(S,\cdot)$,
one says that $A$ is an $IP_r$ set if and only if there exists a sequence $\langle x_t\rangle_{t=1}^r$
in $S$ such that $FP(\langle x_t\rangle_{t=1}^r)\subseteq A$ 
where $FP(\langle x_t\rangle_{t=1}^r)=\{\prod_{t\in H}x_t:H\in\mathcal{P}_f(\nhat{r})\}$
and, if $S$ is not commutative, one specifies that the products are taken in
increasing order of indices.

In the case of $(\mathcal{P}_f(\ben),\cup)$, one standardly adds the condition that\\ 
$FU(\langle X_t\rangle_{t=1}^r)\subseteq A$ where
 $$\textstyle FU(\langle X_t\rangle_{t=1}^r)=\{\bigcup_{t\in H}X_t:H\in\mathcal{P}_f(\nhat{r})\}$$
and if $t<r$, then $\max X_t<\min X_{t+1}$. In the partial semigroup
$(\mathcal{P}_f(\ben),\uplus)$ that is simply the requirement from the 
following natural definition.
 
\begin{definition} Let $(S,*)$ be a partial semigroup and
let $A\subseteq S$. \begin{itemize}
\item[(1)] For $r\in\ben$, $A$ is an $IP_r$ set in $S$ if and only if there
exists a sequence $\langle x_t\rangle_{t=1}^r$
in $S$ such that $FP(\langle x_t\rangle_{t=1}^r)\subseteq A$
where $FP(\langle x_t\rangle_{t=1}^r)=\{\prod_{t\in H}x_t:H\in\mathcal{P}_f(\nhat{r})\}$
and for each $H\in\mathcal{P}_f(\nhat{r})$ the product is taken in increasing order of
indices and is defined.
\item[(2)] For $r\in\ben$, $A$ is an $IP_r^{*}$ set if and only if it has nonempty
intersection with each $IP_r$ set in $S$.
\end{itemize}
\end{definition}

\begin{lemma}\label{FUthm} Let $\mathcal{P}_f(\ben)$ be finitely colored.
There exists a sequence $\langle X_n\rangle_{n=1}^\infty$ in $\mathcal{P}_f(\ben)$
such that for each $n\in\ben$, $\max X_n<\min X_{n+1}$ and
$\{\bigcup_{n\in H}X_n:H\in\mathcal{P}_f(\ben)\}$ is monochromatic.
\end{lemma}

\begin{proof} \cite[Corollary 5.17]{HS}.
\end{proof}

The above lemma can be more succinctly stated in terms of partial semigroups
as {\it Whenever $\mathcal{P}_f(\ben)$ is finitely colored, there is a 
monochromatic IP set in $(\mathcal{P}_f(\ben),\uplus)$.}

The proof of the next lemma is a standard compactness argument. For
$r\in\ben$, let ${\mathcal F}_r=\mathcal{P}_f(\nhat{r})$. Then $({\mathcal F}_r,\uplus)$
is a partial semigroup (but not an adequate partial semigroup).

\begin{lemma}\label{finFUa} Let $s,k\in\ben$. There exists $r\in\ben$ such that
whenever ${\mathcal F}_r$ is $k$-colored, there is a monochromatic
$IP_s$ set in $({\mathcal F}_r,\uplus)$.\end{lemma}

\begin{proof} Suppose the conclusion fails. For each $r\in\ben$ choose\\
$c_r:{\mathcal F}_r\to\nhat{k}$ such that there is no monochromatic
$IP_s$ set in $({\mathcal F}_r,\uplus)$.

Given any $r$, the set of functions from ${\mathcal F}_r$ to $\nhat{k}$
is finite. Choose an infinite subset $B_1$ of $\ben$ such that
for any $a,b\in B_1$, $c_a$ and $c_b$ agree on ${\mathcal F}_1$.
Inductively, given $m\in\ben\setminus \{1\}$ and infinite
$B_{m-1}$, choose an infinite set $B_m\subseteq B_{m-1}$ with
$\min B_m\geq m$ such that for any $a,b\in B_m$, 
$c_a$ and $c_b$ agree on ${\mathcal F}_m$.

Define $d:\mathcal{P}_f(\ben)\to\nhat{k}$ as follows. For $F\in\mathcal{P}_f(\ben)$,
let $m=\max F$, pick $t\in B_m$ and let $d(F)=c_t(F)$. (So for
all $a\in B_m$, $d(F)=c_a(F)$.) Pick by Lemma \ref{FUthm},
a sequence $\langle X_n\rangle_{n=1}^\infty$ in $\mathcal{P}_f(\ben)$
such that for each $n\in\ben$, $\max X_n<\min X_{n+1}$ and
$\{\bigcup_{n\in H}X_n:H\in\mathcal{P}_f(\ben)\}$ is monochromatic with respect to
$d$. 

Let $m=\max X_s$ and pick $a\in B_m$. Then $c_a$ is 
monochromatic on\\ $FU(\langle X_t\rangle_{t=1}^s)$, a contradiction.
\end{proof}

\begin{lemma}\label{finFUb} Let $s,k\in\ben$. There exists
$r\in\ben$ such that whenever $A$ is an $IP_r$ set in $(\mathcal{P}_f(\ben),\uplus)$
and $A$ is $k$-colored, there exists a monochromatic $IP_s$ set $B\subseteq A$.
\end{lemma}

\begin{proof} Pick $r$ as guaranteed by Lemma \ref{finFUa}. Let
$A=FU(\langle H_j\rangle_{j=1}^r)$ be an $IP_r$ set in $(\mathcal{P}_f(\ben),\uplus)$ and let
$c:A\to\nhat{k}$. Define
$d:{\mathcal F}_r\to\nhat{k}$ by for $F\in{\mathcal F}_r$,
$d(F)=c(\bigcup_{j\in F} H_j)$. Pick
$\langle F_i\rangle_{i=1}^s$ such that $FU(\langle F_i\rangle_{i=1}^s)$
is an $IP_s$ set in $({\mathcal F}_r,\uplus)$ on which $d$ is constant
and let $m$ be that constant value. For $i\in\nhat{s}$, let
$K_i=\bigcup_{j\in F_i} H_j$. To see that 
$c$ is constantly equal to $m$ on $B=FU(\langle K_i\rangle_{i=1}^s)$,
let $L\in\mathcal{P}_f(\nhat{s})$. Then\\ $m=d(\bigcup_{i\in L}F_i)=
c(\bigcup\{H_j:j\in\bigcup_{i\in L}F_i\})=
c(\bigcup_{i\in L}K_i)$.
\end{proof}

The proof of the following lemma is based on the proofs
of \cite[Propositions 2.4 and 2.5]{BR}.

\begin{lemma}\label{IPrstarsstar} Let $r,s\in\ben$. There exists 
$q\in\ben$ such that if $A$ is an $IP_r^*$ set in $(\mathcal{P}_f(\ben),\uplus)$
and $B$ is an $IP_s^*$ set in $(\mathcal{P}_f(\ben),\uplus)$, then
$A\cap B$ is an $IP_q^*$ in $(\mathcal{P}_f(\ben),\uplus)$. \end{lemma}

\begin{proof} Assume that $A$ is an $IP_r^*$ set in $(\mathcal{P}_f(\ben),\uplus)$
and $B$ is an $IP_s^*$ set in $(\mathcal{P}_f(\ben),\uplus)$ and assume without 
loss of generality that $s\geq r$. 

Pick by Lemma \ref{finFUb}
some $q\in\ben$ such that whenever $C$ is an $IP_q$ set in $(\mathcal{P}_f(\ben),\uplus)$
and $C$ is 2-colored, there is a monochromatic $IP_s$ set $D\subseteq C$.
We claim: 

\hbox to \hsize{(*)\hfill If $C$ is an $IP_q$ set
and $E$ is an $IP_s^*$ set,
then $C\cap E$ is an $IP_s$ set.\hfill}

To establish (*), let $C$ be an $IP_q$ set and let 
$E$ be an $IP_s^*$ set. Define $c:C\to\{1,2\}$ by $c(x)=1$ if and only if
$x\in C\cap E$. Pick an $IP_s$ set $D\subseteq C$ which is monochromatic 
with respect to $c$. Pick $x\in E\cap D$. Then $c(x)=1$ so $D\subseteq C\cap E$.

To see that $A\cap B$ is an $IP_q^*$ set, let $C$ be an $IP_q$ set.
By (*), $B\cap C$ is an $IP_s$ set. Since $s\geq r$, $B\cap C$ is an $IP_r$
set so $(A\cap B)\cap C=A\cap (B\cap C)\neq\emp$ as required.
\end{proof}

The following crucial lemma is based on \cite[Lemma 2]{G}. It does not 
apply to arbitrary partial semigroups because the functions $g_f$ as
defined in the proof of \cite[Lemma 2]{G} are not likely to be adequate 
sequences.

We write $A=\{x_1,x_2,\ldots, x_m\}_<$ to abbreviate the statement\\
``$A=\{x_1,x_2,\ldots, x_m\}$ and $x_1<x_2<\ldots<x_m$.''

\begin{lemma} \label{lemTheta} Let $(S,*)$ be an adequate partial semigroup
with a two sided identity $e$, let $A$ be a CR set in $S$,
let $F\in\mathcal{P}_f({\mathcal T})$, and let $L\in\mathcal{P}_f(S)$.
Let $\Theta=\Theta(A,F,L)=\{M\in\mathcal{P}_f(\ben):\hbox{if }m=|M|\hbox{ and }
M=\{t(1),t(2),\ldots,t(m)\}_<\,,\\
\hbox{then }
(\exists a\in S^{m+1})(\forall f\in F)\\
(a(1)*f(t(1))*a(2)*\ldots
*a(m)*f(t(m))*a(m+1)\in A\cap\sigma(L))\}$.
If $k\in\ben$, $|F|\leq k$, and $r=r(A,k,L)$, then $\Theta$ is an $IP_r^*$ 
set in $(\mathcal{P}_f(\ben),\uplus)$.
\end{lemma}

\begin{proof} Assume that $k\in\ben$, $|F|\leq k$, and $r=r(A,k,L)$. 
To see that $\Theta$ is an $IP_r^*$ set, let $B$ be an
$IP_r$ set in $(\mathcal{P}_f(\ben),\uplus)$ and pick $\langle H_n\rangle_{n=1}^r$ in $\mathcal{P}_f(\ben)$ such that 
$\max H_n<\min H_{n+1}$ for $n\in\nhat{r-1}$ and $FU(\langle H_n\rangle_{n=1}^r)\subseteq B$.

For $n\in\nhat{r}$ let $\alpha_n=|H_n|$ and write
$$H_n=\{b(n,1),b(n,2),\ldots,b(n,\alpha_n)\}_<\,.$$ 
Let $z=\max H_r=b(r,\alpha_r)$.
For $f\in F$, define $g_f\in\ntos$ by, for $n\in\nhat{r}$,
$$g_f(n)=f(b(n,1))*e*f(b(n,2))*e*\ldots*e*f(b(n,\alpha_n))$$
and $g_f(n)=f(z+n)$ if $n>r$.  It is routine to verify that each $g_f\in{\mathcal T}$.

Now $\{g_f:f\in F\}\in \mathcal{P}_f(\mathcal{T})_{\leq k}$.  Pick $m\in\ben$,
$a\in S^{m+1}$, and $t(1)<t(2)<\ldots <t(m)\leq r$ in $\ben$ such that
for all $f\in F$,\\
$a(1)*g_f(t(1))*a(2)*\ldots* a(m)*g_f(t(m))*a(m+1)\in A\cap\sigma(L)$.

We claim that $\bigcup_{j=1}^m H_{t(j)}\in \Theta$, so that $\Theta\cap B\neq\emp$.

Let $p=|\bigcup_{j=1}^m H_{t(j)}|=\sum_{j=1}^m\alpha_j$. Then the elements
of $\bigcup_{t=1}^m H_{t(j)}=\{s(1),s(2),\ldots,s(p)\}_<$ listed in increasing order are:

$\begin{array}{l}
b(t(1),1),b(t(1),2),\ldots,b(t(1),\alpha_{t(1)}),\\
b(t(2),1),b(t(2),2),\ldots,b(t(2),\alpha_{t(2)}),\\
\hskip 50 pt \vdots\\
b(t(m),1),b(t(m),2),\ldots,b(t(m),\alpha_{t(m)})\,.\end{array}$

Define $c\in S^{p+1}$ by $c(p+1)=a(m+1)$ and for $i\in\nhat{p}$,
$$c(i)=\left\{\begin{array}{cl}
a(j)&\hbox{if }j\in\nhat{m}\hbox{ and }s(i)=b(t(j),1)\\
e&\hbox{if }s(i)\notin\{b(t(1),1),b(t(2),1),\ldots,b(t(m),1)\}\,.\end{array}\right.$$

Then for each $f\in F$, 
$$\begin{array}{l}
c(1)*f(s(1))*c(2)*\ldots*c(p)*f(s(p))*c(p+1)={}\\
a(1)*g_f(t(1))*a(2)*\ldots* a(m)*g_f(t(m))*a(m+1)\in A\cap\sigma(L)\,.
\end{array}$$
So $\bigcup_{j=1}^m H_{t(j)}\in \Theta$ as claimed.
\end{proof}

\begin{theorem}\label{thmprod} Let $(S,*)$ and $(T,*)$ be adequate partial
semigroups, each with a two sided identity, let $A$ be a CR set in $S$,
and let $B$ be a CR set in $T$. Then $A\times B$ is a CR set in $S\times T$.
\end{theorem}

\begin{proof} Let $k\in \ben$. To see that $A\times B$ is a $k$-CR set, 
let $L\in\mathcal{P}_f(S\times T)$. Let $F\in\mathcal{P}_f({\mathcal T}_{S\times T})$ with $|F|\leq k$.
Let $u=r(A,k,\pi_1[L])$ and let $v=r(B,k,\pi_2[L])$.
Pick by Lemma \ref{IPrstarsstar} some $q\in\ben$ such that whenever
$C$ is an $IP_u^*$ set in $(\mathcal{P}_f(\ben),\uplus)$
and $D$ is an $IP_v^*$ set in $(\mathcal{P}_f(\ben),\uplus)$, one has that
$C\cap D$ is an $IP_q^*$ set in $(\mathcal{P}_f(\ben),\uplus)$.

We shall show that there exist $m\in\ben$, $c\in(S\times T)^{m+1}$, and
$t(1)<t(2)<\ldots <t(m)\leq q$ in $\ben$ such that for each $f\in F$,
$$c(1)*f(t(1))*c(2)*\ldots*c(m)*f(t(m))*c(m+1)\in(A\times B)\cap\sigma(L)\,.$$
Let $G=\{\pi_1\circ f:f\in F\}$ and let $H=\{\pi_2\circ f:f\in F\}$.
Let $\Theta_1=\Theta(A,G,\pi_1[L])$ and $\Theta_2=\Theta(B,H,\pi_2[L])$ be
as defined in Lemma \ref{lemTheta}. Then by that lemma,
$\Theta_1$ is an $IP_u^*$ set in $(\mathcal{P}_f(\ben),\uplus)$ and
$\Theta_2$ is an $IP_v^*$ set in $(\mathcal{P}_f(\ben),\uplus)$ so 
$\Theta_1\cap\Theta_2$ is an $IP_q^*$ set in $(\mathcal{P}_f(\ben),\uplus)$
and thus $\Theta_1\cap \Theta_2\cap FU(\langle\{i\}\rangle_{i=1}^q)\neq\emp$.
Note that $FU(\langle\{i\}\rangle_{i=1}^q)=\mathcal{P}_f(\nhat{q}$.
Pick $R\in \Theta_1\cap \Theta_2\cap \mathcal{P}_f(\nhat{q})$ and let $m=|R|$.
Then $R=\{t(1),t(2),\ldots,t(m)\}_<$ and $t(m)\leq q$.
Since $R\in \Theta_1$, pick $a\in S^{m+1}$ such that for all $f\in F$,
$$a(1)*\pi_1(f(t(1)))*a(2)*\ldots*a(m)*\pi_1(f(t(m)))*a(m+1)\in A\cap \sigma(\pi_1[L])\,.$$
Since $R\in \Theta_2$, pick $b\in T^{m+1}$ such that for all $f\in F$,
$$b(1)*\pi_2(f(t(1)))*b(2)*\ldots*b(m)*\pi_2(f(t(m)))*b(m+1)\in B\cap \sigma(\pi_2[L])\,.$$

For $i\in\nhat{m}$ let $c(i)=(a(i),b(i))$. 
Then 
$$\begin{array}{l}
c(1)*f(t(1))*c(2)*\ldots*c(m)*f(t(m))*c(m+1)\in{}\\
 (A\times B)\cap(\sigma(\pi_1[L])\times
\sigma(\pi_2[L]))\subseteq(A\times B)\cap\sigma(L)\,.\end{array}$$ 
Thus $A\times B$ is a $k$-CR set.
\end{proof}

\section{Some simpler descriptions}\label{secsimpler}

It was shown in \cite[Lemma 14.14.2]{HS} that in commutative semigroups,
the definition of a J set can be replaced by the obvious simplified version.
Similarly in \cite[Lemma 2.3]{HHST} the corresponding result for CR sets was established.
We show in this section that if $(S,*)$ is an adequate partial semigroup, then
the corresponding simplification can be made for CR sets. It is some what surprising that we 
are not able to obtain the same result for J sets, since generally J sets have been easier 
to handle.  

\begin{lemma}\label{fpsigma} Let $(S,*)$ be an infinite adequate partial semigroup
and let $F\in\mathcal{P}_f(S)$. There exists a sequence $\seqx$ in $S$ such that
$FP(\seqx)\subseteq \sigma(F)$. In fact, if $p$ is an idempotent in
$\delta S$ and $A\in p$, then there exists a sequence $\seqx$ in $S$ such that
$FP(\seqx)\subseteq A$.\end{lemma}

\begin{proof} This follows by a now standard argument from \cite[Lemma 2.12]{HM}.
Two different proofs are given in \cite[Theorem 4.6]{M}. \end{proof}

The following theorem was proved in \cite[Theorem 4.1]{HP} for the case when 
$S$ is countable, without the assumption that ${\mathcal T}\neq\emp$. (So that
result establishes that for any countable $S$, ${\mathcal T}\neq\emp$.)

\begin{theorem}\label{adeqxtend} Let $(S,*)$ be an adequate partial semigroup, let
$n\in\ben$, and let $\langle f(t)\rangle_{t=1}^n$ be a sequence in $S$ with the
property that whenever $\emp\neq H\subseteq\nhat{n}$, $\prod_{t\in H}f(t)$
is defined. If ${\mathcal T}\neq\emp$, then $f$ can be extended to an
adequate sequence $\langle f(t)\rangle_{t=1}^\infty$.\end{theorem}

\begin{proof} Assume that ${\mathcal T}\neq\emp$ and pick $g\in{\mathcal T}$.
Let $$\textstyle M=\big\{\prod_{t\in H}f(t):\emp\neq H\subseteq\nhat{n}\big\}\,.$$
Pick $m\in\ben$ such that $FP(\langle g(t)\rangle_{t=m}^\infty)\subseteq\sigma(M)$.
For $j\in\ben$, let $f(n+j)=g(m+j)$. To see that for all $H\in\mathcal{P}_f(\ben)$,
$\prod_{t\in H}f(t)$ is defined, let $H\in\mathcal{P}_f(\ben)$. If $\max H\leq n$, then
$\prod_{t\in H}f(t)$ is defined by assumption. If $\min H>n$, then
$\prod_{t\in H}f(t)=\prod_{s\in H+m-n}g(s)$. So assume that
$\min H\leq n<\max H$, let $L=H\cap\nhat{n}$, and let $G=H\setminus\nhat{n}$.
Then $\prod_{t\in L}f(t)\in M$ and $\prod_{t\in G}f(t)=\prod_{s\in G+m-n}g(s)\in\sigma(M)$
so $\prod_{t\in H}f(t)=(\prod_{t\in L}f(t))*(\prod_{t\in G}f(t))$ is defined.

Now let $F\in\mathcal{P}_f(S)$ be given and pick $r\in\ben$ such that 
$FP(\langle g(t)\rangle_{t=r}^\infty)\subseteq\sigma(F)$. We may presume that 
$r\geq m$. Then $$FP(\langle f(t)\rangle_{t=n+r-m}^\infty)=FP(\langle g(s)\rangle_{s=r}^\infty)\subseteq\sigma(F).$$
\end{proof}

\begin{theorem}\label{CRsimpler} Let $(S,*)$ be a commutative adequate partial semigroup,
let $A\subseteq S$, and let $k\in\ben$. The following statements are 
equivalent.
\begin{itemize}
\item[(1)] $(\forall L\in\mathcal{P}_f(S))(\exists r\in\ben)(\forall F\in\mathcal{P}_f(\mathcal{T})_{\leq k})
(\exists b\in S)(\exists H\in\mathcal{P}_f(\nhat{r})) \\
(\forall f\in F)(b*\prod_{t\in H}f(t)\in A\cap\sigma(L))$.
\item[(2)] $(\forall L\in\mathcal{P}_f(S))(\exists r\in\ben)\\
(\forall F\in\mathcal{P}_f(\mathcal{T})_{\leq k})(\exists m\in\ben)
(\exists a\in S^{m+1})(\exists t(1)<\ldots <t(m)\leq r\hbox{ in }\ben)\\
(\forall f\in F)
(a(1)*f(t(1))*a(2)*\ldots* a(m)*f(t(m))*a(m+1)\in A\cap\sigma(L))$.
\end{itemize}\end{theorem}

\begin{proof} Note that if ${\mathcal T}=\emp$, both statements are vacuously true.

(1) implies (2). Let $L\in\mathcal{P}_f(S)$  and assume we have $r\in\ben$ such that
whenever $G\in\mathcal{P}_f({\mathcal T})_{\leq k}$, there exist $b\in S$
and $H\in\mathcal{P}_f(\nhat{r})$ such that for all $f\in G$, $b\cdot\prod_{t\in H}f(t)\in A\cap\sigma(L)$.

Let $F\in\mathcal{P}_f(\mathcal{T})_{\leq k}$ be given. 
Let $$\textstyle M=\{\prod_{t\in H}f(t):f\in F\hbox{ and }\emp\neq H\subseteq\nhat{r}\}\,.$$
By Lemma \ref{fpsigma} pick $\langle h(j)\rangle_{j=1}^r$ 
such that $FP(\langle h(j)\rangle_{j=1}^r)\subseteq \sigma(M)$.
For $f\in F$ define $g_f:\nhat{r}\to S$ by for $j\in\nhat{r}$, $g_f(j)=f(j)*h(j)$.

Given $\emp\neq H\subseteq\nhat{r}$, $\prod_{t\in H}f(t)\in M$ so
$\prod_{t\in H}f(t)*\prod_{t\in H}h(t)$ is defined so 
$\prod_{t\in H}g_f(t)=\prod_{t\in H}(f(t)*h(t))$ is defined. By Theorem \ref{adeqxtend}
each $g_f$ can be extended to an adequate sequence which we also denote by 
$g_f$.

Let $G=\{g_f:f\in F\}$. Then $G\in\mathcal{P}_f({\mathcal T})_{\leq k}$ so pick $b\in S$
and $H\in\mathcal{P}_f(\nhat{r})$ such that for all $f\in F$, $b*\prod_{t\in H}g_f(t)\in A\cap\sigma(L)$.
Let $m=|H|$, let $\langle t(1),t(2),\ldots,t(m)\rangle$ enumerate $H$ in increasing order,
let $a(1)=b$, and for $j\in\{2,3,\ldots,m+1\}$ let $a(j)=h(t(j-1))$.
Then for each $f\in F$, 
$a(1)*f(t(1))*a(2)*\ldots* a(m)*f(t(m))*a(m+1)=
b*\prod_{t\in H}g_f(t)\in A\cap\sigma(L)$. 

(2) implies (1). Let $L\in\mathcal{P}_f(S)$. Pick $r\in\ben$ such that
$(\forall F\in\mathcal{P}_f(\mathcal{T})_{\leq k})\\
(\exists m\in\ben)
(\exists a\in S^{m+1})(\exists t(1)<\ldots <t(m)\leq r\hbox{ in }\ben)\\
(\forall f\in F)
(a(1)*f(t(1))*a(2)*\ldots* a(m)*f(t(m))*a(m+1)\in A\cap\sigma(L))$.

Let $b=\prod_{j=1}^{m+1}a(j)$ and let $H=\{t(1),t(2),\ldots,t(m)\}$.
For each $f\in F$,  $b*\prod_{t\in H}f(t)=
a(1)*f(t(1))*a(2)*\ldots* a(m)*f(t(m))*a(m+1)\in A\cap\sigma(L)$.
\end{proof}

The following similar result for J sets in countable adequate
partial semigroups was proved in \cite{HP}.

\begin{theorem}\label{Jsimpler} Let $(S,*)$ be a countable commutative adequate partial semigroup,
and let $A\subseteq S$. The following statements are 
equivalent.
\begin{itemize}
\item[(1)] $(\forall L\in\mathcal{P}_f(S))(\forall F\in\mathcal{P}_f(\mathcal{T}))
(\exists b\in S)(\exists H\in\mathcal{P}_f(\ben))\\
(\forall f\in F)(b*\prod_{t\in H}f(t)\in A\cap\sigma(L))$.
\item[(2)] $(\forall L\in\mathcal{P}_f(S))
(\forall F\in\mathcal{P}_f(\mathcal{T}))(\exists m\in\ben)
(\exists a\in S^{m+1})\\
(\exists t(1)<\ldots <t(m)\hbox{ in }\ben)
(\forall f\in F)\\
(a(1)*f(t(1))*a(2)*\ldots* a(m)*f(t(m))*a(m+1)\in A\cap\sigma(L))$.
\end{itemize}\end{theorem}

\begin{proof} It is trivial that (2) implies (1). That (1) implies (2) is \cite[Theorem 4.4]{HP}.
\end{proof}

We are unable to remove the countability assumption 
from Theorem \ref{Jsimpler}.

\vspace{0.3cm}

\noindent   \textbf{Acknowledgement:} The first named author would like to thank the Department of Mathematics, University of Haifa for her position.


\begin{thebibliography}{}

\bibitem{BBH} V. Bergelson, A. Blass, and N. Hindman, {\it Partition theorems for spaces of
variable words\/}, Proc.\ London Math.\ Soc.\  {\bf 68} (1994), 449-476.

\bibitem{BG} V. Bergelson and D. Glasscock, {\it On the interplay between additive 
and multiplicative largeness and its combinatorial applications\/} J. Combin.\ Theory Ser.\ A 
{\bf 172} (2020), 105203.

\bibitem{BR} V. Bergelson and D. Robertson, {\it Polynomial recurrence with large 
intersection over countable fields\/}, Israel J. Math.\ {\bf 214} (2016), 109-120.

\bibitem{G} S. Goswami, {\it Cartesian product of two CR sets\/}, Semigroup Forum
{\bf 108} (2024), 759-763.

\bibitem{H} N. Hindman, {\it Notions of size in a semigroup --
an update from a historical perspective\/}, Semigroup Forum
{\bf 100} (2020), 52-76.

\bibitem{HHST} N. Hindman, H. Hosseini, D. Strauss, and M. 
Tootkaboni, {\it Combinatorially rich sets
in arbitrary semigroups\/}, Semigroup Forum {\bf 107} (2023), 127-143.

\bibitem{HM} N. Hindman and R. McCutcheon, {\it VIP systems in partial semigroups\/},
Discrete Math.\ {\bf 240} (2001), 45-70. 

\bibitem{HP} N. Hindman and  K. Pleasant, 
{\it Central Sets Theorem for arbitrary adequate partial semigroups\/},
Topology Proceedings {\bf 58} (2021), 183-206.

\bibitem{HS} N. Hindman and D. Strauss, {\it Algebra in the
Stone-\v Cech compactification: theory and applications, second edition\/},
de Gruyter, Berlin, 2012.

\bibitem{M} J. McLeod, {\it Notions of size in adequate partial semigroups\/},
Ph.\ D. Thesis, Howard University, 2001.

\end{thebibliography}
\end{document}